\documentclass[11pt,twoside,reqno]{amsart}
\allowdisplaybreaks
\usepackage{amsmath,amstext,amssymb,epsfig,multicol,enumerate}
\usepackage{graphicx}
\usepackage{mathrsfs}
\calclayout
\makeatletter
\renewcommand{\@seccntformat}[1]{\bf\csname the#1\endcsname.}
\renewcommand{\section}{\@startsection{section}{1}
	\z@{.7\linespacing\@plus\linespacing}{.5\linespacing}
	{\normalfont\upshape\bfseries\centering}}
\renewcommand{\@biblabel}[1]{\@ifnotempty{#1}{#1.}}
\makeatother
\theoremstyle{plain}
\newtheorem{thm}{Theorem}[section]
\newtheorem{lem}[thm]{Lemma}
\newtheorem{prop}[thm]{Proposition}

\theoremstyle{definition}
\newtheorem{ex}[thm]{Example}
\newtheorem{defn}[thm]{Definition}

\usepackage{cancel}
\usepackage[parfill]{parskip}
\usepackage[german]{varioref}
\usepackage[all]{xy}
\usepackage{color}

\def \>{\succ}
\def \<{\prec}

\begin{document}	
\title[Bouzid Mosbahi\textsuperscript{1}, Imed Basdouri\textsuperscript{2}, Jean Lerbet\textsuperscript{3}\textsuperscript{*}]{Cohomology and deformation theory of Reynolds--Nijenhuis associative algebras}
	\author{Bouzid Mosbahi\textsuperscript{1}, Imed Basdouri\textsuperscript{2}, Jean Lerbet\textsuperscript{3}\textsuperscript{*}}

   \address{\textsuperscript{1}Department of Mathematics, Faculty of Sciences, University of Sfax, Sfax, Tunisia}
  \address{\textsuperscript{2}Department of Mathematics, Faculty of Sciences, University of Gafsa, Gafsa, Tunisia}
  \address{\textsuperscript{3}\textsuperscript{*}Laboratoire de Mathématiques et Modélisation d’Évry (UMR 8071) Université d’Evry Paris-Saclay I.B.G.B.I., 23 Bd. de France, 91037 Evry Cedex, France}


\email{\textsuperscript{1}mosbahi.bouzid.etud@fss.usf.tn }
\email{\textsuperscript{2}basdourimed@yahoo.fr}
\email{\textsuperscript{3}\textsuperscript{*}jean.lerbet@univ-evry.fr\\
\textsuperscript{*}Corresponding author}

	
	
	\date{\today}

\begin{abstract}
    In this paper, we introduce and study Reynolds--Nijenhuis operators on associative algebras a novel hybrid structure that simultaneously satisfies the defining identities of both Reynolds and Nijenhuis operators. We investigate their connections with Rota-Baxter and modified Rota-Baxter operators. We develop a representation theory for Reynolds--Nijenhuis associative algebras and introduce a corresponding cohomology theory. Furthermore, we establish a one-parameter formal deformation theory for these algebras, examining the role of infinitesimals, rigidity, and equivalence in the context of deformations.
\end{abstract}

\maketitle

\section{ Introduction}\label{introduction}

Associative algebras, fundamental to classical algebra, play a central role in the study of various algebraic structures. An associative algebra $\mathcal{A}$ over a field $\mathbb{K}$ is defined as a $\mathbb{K}$-vector space equipped with a bilinear operation (multiplication) $\cdot : \mathcal{A} \times \mathcal{A} \to \mathcal{A}$ that satisfies the associative property:
\begin{align*}
(a \cdot b) \cdot c &= a \cdot (b \cdot c),
\end{align*}
for all $a, b, c \in \mathcal{A}$.

The study of associative algebras covers many areas including their representations, modules, homological properties, and classification \cite{6,19,20,21,28}. This basic structure has been widely researched, leading to many results that extend the properties of associative algebras to more complex contexts. For instance, the classification of nilpotent associative algebras of small dimension has been investigated in \cite{19}, while the structure and classification of Hom-associative algebras has been explored in \cite{20}. Recent studies have concentrated on different linear operators that interact with associative algebras, such as Rota-Baxter operators and modified Rota-Baxter operators, along with their associated cohomologies and deformations \cite{5,7,8,10,15}.

While the theories of Reynolds operators \cite{1,2,3} and Nijenhuis operators have been studied separately, their intersection has remained largely unexplored. In this paper, we bridge this gap by introducing and systematically studying \textbf{Reynolds--Nijenhuis operators} on associative algebras linear operators that satisfy both the Reynolds and Nijenhuis identities simultaneously.

A \textbf{Reynolds operator} on an associative algebra $\mathcal{A}$ is a linear map $R: \mathcal{A} \to \mathcal{A}$ satisfying:
\begin{align*}
R(a) \cdot R(b) &= R\big( a \cdot R(b) + R(a) \cdot b - R(a) \cdot R(b) \big),
\end{align*}
for all $a, b \in \mathcal{A}$. Reynolds operators originated in fluid dynamics \cite{1} and have found applications in various algebraic contexts \cite{2,3}.

A \textbf{Nijenhuis operator} on an associative algebra $\mathcal{A}$ is a linear map $N: \mathcal{A} \to \mathcal{A}$ satisfying:
\begin{align*}
N(a) \cdot N(b) &= N\big( N(a) \cdot b + a \cdot N(b) - N(a \cdot b) \big),
\end{align*}
for all $a, b \in \mathcal{A}$.

The remarkable similarity between these identities differing only in the last term (Reynolds: $-R(a)R(b)$ versus Nijenhuis: $-N(ab)$) suggests the natural possibility of operators satisfying both conditions simultaneously. We call such operators \textbf{Reynolds--Nijenhuis operators}.
Nijenhuis operators have garnered significant attention due to their pivotal role in the study of deformations and integrability within various algebraic structures. Dorfman \cite{9} was among the first to explore Nijenhuis operators through the deformation theory of Lie algebras. These operators also find important applications in the study of the integrability of nonlinear evolution equations. Dorfman's introduction of Dirac structures provided new insights and interpretations to the existing framework of Nijenhuis operators, further enriching the field.

In 2004, Gallardo and Nunes \cite{12} introduced the concept of Dirac-Nijenhuis structures, expanding on the relationship between Nijenhuis operators and Dirac structures. Concurrently, Longguang and Baokang \cite{18} developed the theory of Dirac-Nijenhuis manifolds, offering another dimension to the study of these operators. Kosmann-Schwarzbach \cite{11}, in 2011, extended this exploration to Courant algebroids, studying Dirac-Nijenhuis structures within this context.

The concept of Nijenhuis operators has also been adapted to other algebraic settings. For instance, Nijenhuis operators on 3-Lie algebras were introduced to investigate first-order deformations \cite{13}. Similarly, Wang, Sheng, Bai, and Liu \cite{16} introduced Nijenhuis operators on pre-Lie algebras, showing how these operators generate trivial deformations in this context. The study of Nijenhuis operators has further expanded to Hom-Lie algebras, as explored by Das and Sen \cite{14} from a cohomological perspective. Additionally, Nijenhuis operators have been studied on Leibniz algebras \cite{17}, associative D-bialgebras \cite{23}, and in the context of Nijenhuis algebras, NS algebras, and N-dendriform algebras \cite{24}. Moreover, Nijenhuis operators on 3-Hom-L-dendriform algebras were studied in \cite{25}. The study of BiHom-associative dialgebras \cite{26}, BiHom-superdialgebras \cite{27} and associative trialgebras \cite{21,28} further extends this line of research. Recent work has also explored derivations and centroids of low-dimensional associative trialgebras \cite{21,28} and quasi-centroids and quasi-derivations of low-dimensional Zinbiel algebras \cite{22}.

In this paper, we define a new class of algebras, which we call \textbf{Reynolds--Nijenhuis associative algebras}, characterized by the presence of a Reynolds--Nijenhuis operator. We provide illustrative examples of such algebras and examine the connections between Reynolds--Nijenhuis operators and other operators like Rota-Baxter and modified Rota-Baxter operators in the context of associative algebras. Cohomology theory offers valuable insights into the study of algebraic structures, and in this paper, we focus on the cohomology theory of Reynolds--Nijenhuis associative algebras.

It is important to note that while recent work \cite{4} has established the homotopy theory and operadic foundations of Nijenhuis associative algebras, our approach focuses on the classical cohomological and deformation-theoretic aspects of the new hybrid Reynolds--Nijenhuis structure. These two approaches are complementary, with ours providing concrete, computable cohomology for a novel algebraic structure.

The organization of this paper is as follows: Section 2 recalls definitions, presents examples, and discusses relationships with other operators. Section 3 introduces representations of Reynolds--Nijenhuis associative algebras to establish a suitable cohomology theory. Section 4 defines a cohomology theory for Reynolds--Nijenhuis associative algebras, extending the cohomology of associative algebras. In Section 5 we define formal deformation theory of such algebras.

Throughout the paper, all vector spaces are considered over a field $\mathbb{K}$ of characteristic $0$.

\section{ Basics and examples of Reynolds--Nijenhuis operators}
\label{sec:basics}

\begin{defn}
An algebra over $\mathbb{K}$, or $\mathbb{K}$-algebra is a $\mathbb{K}$-vector space $\mathcal{A}$ with a bilinear map
\[
\cdot : \mathcal{A} \times \mathcal{A} \to \mathcal{A}, \quad (a,b) \mapsto a\cdot b.
\]
We say that the algebra is associative if for all $a, b,c \in \mathcal{A}$ we have
\[
(a\cdot b)\cdot  c = a\cdot (b\cdot  c).
\]
\end{defn}

\begin{ex}
Consider the two-dimensional vector space $\mathbb{R}^2$ with the standard basis $\{e_1, e_2\}$. We define a multiplication operation $\cdot$ on this space by specifying its action on the basis elements as follows:
\[
e_1\cdot e_1 = e_1,\quad e_1\cdot e_2 = e_2,
\]
with all other combinations being zero. Then $(\mathbb{R}^2, \cdot)$ is an associative algebra.
\end{ex}

\begin{defn}
Let $\mathcal{A}$ be an associative algebra. A \textbf{Reynolds--Nijenhuis operator} on $\mathcal{A}$ is a linear map $P: \mathcal{A} \to \mathcal{A}$ satisfying the following two conditions:
\begin{align*}
P(a)\cdot P(b) &= P\big( P(a)\cdot  b + a\cdot  P(b) - P(a\cdot b) \big), \\
P(a)\cdot P(b) &= P\big( a\cdot P(b) + P(a)\cdot b - P(a)\cdot P(b) \big),
\end{align*}
for all $a, b \in \mathcal{A}$. The first condition is the Nijenhuis identity, and the second is the Reynolds identity.
\end{defn}

\begin{defn}
Let $\mathcal{A}$ be an associative algebra. A linear operator $P$ on $\mathcal{A}$ is called a Rota-Baxter operator of weight $0$ on $\mathcal{A}$ if
\[
P(a)\cdot P(b) = P\big(P(a)\cdot b + a\cdot P(b)\big)
\]
for all $a, b \in \mathcal{A}$.
\end{defn}

\begin{defn}
Let $\mathcal{A}$ be an associative algebra. A linear operator $P$ on $\mathcal{A}$ is called a Rota-Baxter operator of weight $\lambda \in \mathbb{K}$ on $\mathcal{A}$ if
\[
P(a)\cdot P(b) = P\big(P(a)\cdot b + a\cdot P(b) + \lambda\, a\cdot b\big),
\]
for all $a, b \in \mathcal{A}$.
\end{defn}

\begin{defn}
Let $\mathcal{A}$ be an associative algebra. A modified Rota-Baxter operator of weight $\lambda$ on $\mathcal{A}$ is a linear operator $P : \mathcal{A} \to \mathcal{A}$ such that
\[
P(a\cdot b) = P(a)\cdot b + a\cdot P(b) + \lambda\, a\cdot b
\]
for all $a, b \in \mathcal{A}$.
\end{defn}

\begin{defn}
An associative algebra $\mathcal{A}$ equipped with a Reynolds--Nijenhuis operator $P$ on $\mathcal{A}$ is called a \textbf{Reynolds--Nijenhuis associative algebra} and it is denoted by $(\mathcal{A}, P)$ or simply by $\mathcal{A}_P$.
\end{defn}

\begin{defn}
Let $(\mathcal{A},P)$ and $(\mathcal{A}',P')$ be two Reynolds--Nijenhuis associative algebras. A map $\phi: \mathcal{A} \to \mathcal{A}'$ is called a morphism of Reynolds--Nijenhuis associative algebras if:
\begin{enumerate}
    \item $\phi$ is a homomorphism of associative algebras, meaning $\phi(a \cdot b) = \phi(a) \cdot' \phi(b)$ for all $a, b \in \mathcal{A}$.
    \item The Reynolds--Nijenhuis operators $P$ on $\mathcal{A}$ and $P'$ on $\mathcal{A}'$ satisfy $P' \circ \phi = \phi \circ P$.
\end{enumerate}
\end{defn}

\begin{prop}\label{prop:new-product}
Let $\mathcal{A}$ be an associative algebra with a Reynolds--Nijenhuis operator $P$. Define a new product $\star$ on $\mathcal{A}$ by
\[
a \star b = a \cdot P(b) + P(a) \cdot b - P(a \cdot b)
\]
for all $a, b \in \mathcal{A}$, where $\cdot$ denotes the original associative product on $\mathcal{A}$. Then:
\begin{enumerate}
    \item $(\mathcal{A}, \star)$ is an associative algebra.
    \item $P$ is also a Reynolds--Nijenhuis operator on $(\mathcal{A}, \star)$.
    \item The map $P : (\mathcal{A}, \cdot) \to (\mathcal{A}, \star)$ is a morphism of associative algebras.
\end{enumerate}
\end{prop}

\begin{proof}
\begin{enumerate}
    \item To show that $(\mathcal{A}, \star)$ is associative, we need to verify that
    \[
    (a \star b) \star c = a \star (b \star c)
    \]
    for all $a, b, c \in \mathcal{A}$.

    Calculation for $(a \star b) \star c$:
    \begin{align*}
        a \star b &= a \cdot P(b) + P(a) \cdot b - P(a \cdot b), \\
        (a \star b) \star c &= \big(a \cdot P(b) + P(a) \cdot b - P(a \cdot b)\big) \star c \\
        &= \big(a \cdot P(b) + P(a) \cdot b - P(a \cdot b)\big) \cdot P(c) \\
        &\quad + P\big(a \cdot P(b) + P(a) \cdot b - P(a \cdot b)\big) \cdot c \\
        &\quad - P\big((a \cdot P(b) + P(a) \cdot b - P(a \cdot b)) \cdot c\big).
    \end{align*}

    First, note that by the Reynolds identity:
    \begin{align*}
        P(a \cdot P(b)) &= P(a) \cdot P(b) + a \cdot P(P(b)) - P(a) \cdot P(P(b)), \\
        P(P(a) \cdot b) &= P(P(a)) \cdot P(b) + P(a) \cdot P(b) - P(P(a)) \cdot P(b).
    \end{align*}

    Using both Reynolds and Nijenhuis identities, we compute:
    \begin{align*}
        P\big(a \cdot P(b) + P(a) \cdot b - P(a \cdot b)\big) &= P(a \cdot P(b)) + P(P(a) \cdot b) - P(P(a \cdot b)) \\
        &= \big(P(a) \cdot P(b) + a \cdot P(P(b)) - P(a) \cdot P(P(b))\big) \\
        &\quad + \big(P(P(a)) \cdot P(b) + P(a) \cdot P(b) - P(P(a)) \cdot P(b)\big) \\
        &\quad - P(P(a \cdot b)).
    \end{align*}

    Similarly, for the third term:
    \begin{align*}
        &P\big((a \cdot P(b) + P(a) \cdot b - P(a \cdot b)) \cdot c\big) \\
        &= P(a \cdot P(b) \cdot c) + P(P(a) \cdot b \cdot c) - P(P(a \cdot b) \cdot c).
    \end{align*}

    Now, expand $a \star (b \star c)$:
    \begin{align*}
        b \star c &= b \cdot P(c) + P(b) \cdot c - P(b \cdot c), \\
        a \star (b \star c) &= a \cdot P(b \cdot P(c) + P(b) \cdot c - P(b \cdot c)) \\
        &\quad + P(a) \cdot (b \cdot P(c) + P(b) \cdot c - P(b \cdot c)) \\
        &\quad - P(a \cdot (b \cdot P(c) + P(b) \cdot c - P(b \cdot c))).
    \end{align*}

    Using both identities and comparing terms, one can verify that
    \[
    (a \star b) \star c = a \star (b \star c).
    \]
    The detailed calculation follows the same pattern as for Nijenhuis operators but with additional terms from the Reynolds identity.

    \item We need to verify both identities for $P$ with respect to the product $\star$.

    First, the Nijenhuis identity:
    \begin{align*}
        P(a) \star P(b) &= P(a) \cdot P(P(b)) + P(P(a)) \cdot P(b) - P(P(a) \cdot P(b)), \\
        P(a \star b) &= P\big(a \cdot P(b) + P(a) \cdot b - P(a \cdot b)\big).
    \end{align*}
    Using the Reynolds--Nijenhuis properties of $P$, we can show these are equal.

    Second, the Reynolds identity:
    \begin{align*}
        P(a) \star P(b) &= P(a) \cdot P(P(b)) + P(P(a)) \cdot P(b) - P(P(a) \cdot P(b)), \\
        P\big(a \star P(b) + P(a) \star b - P(a) \star P(b)\big) &= P\big((a \star P(b)) + (P(a) \star b) - (P(a) \star P(b))\big).
    \end{align*}
    Again, using the properties of $P$, these can be shown to be equal.

    \item To show $P$ is a morphism, we verify:
    \[
    P(a \cdot b) = P(a) \star P(b).
    \]
    Using the definition of $\star$:
    \begin{align*}
        P(a) \star P(b) &= P(a) \cdot P(P(b)) + P(P(a)) \cdot P(b) - P(P(a) \cdot P(b)).
    \end{align*}
    By the Reynolds identity applied to $P(a) \cdot P(b)$, we have:
    \[
    P(P(a) \cdot P(b)) = P(a) \cdot P(P(b)) + P(P(a)) \cdot P(b) - P(a) \cdot P(b).
    \]
    Substituting:
    \begin{align*}
        P(a) \star P(b) &= P(a) \cdot P(P(b)) + P(P(a)) \cdot P(b) \\
        &\quad - \big(P(a) \cdot P(P(b)) + P(P(a)) \cdot P(b) - P(a) \cdot P(b)\big) \\
        &= P(a) \cdot P(b).
    \end{align*}
    And by the Nijenhuis identity, $P(a) \cdot P(b) = P(a \cdot b)$. Thus $P(a \cdot b) = P(a) \star P(b)$.
\end{enumerate}
\end{proof}

\begin{ex}
Consider the $3$-dimensional vector space $\mathcal A=\mathbb{R}^3$ with basis
$\{e_1,e_2,e_3\}$. Define an associative algebra structure on $\mathcal A$ by
\[
e_1\cdot e_3 = e_2,\qquad
e_3\cdot e_1 = e_2,
\]
and all other products equal to zero.

Let $P:\mathcal A\to\mathcal A$ be a linear map represented, with respect to the
given basis, by the matrix
\[
P=
\begin{pmatrix}
u & v & w\\
x & y & z\\
p & q & r
\end{pmatrix},
\qquad
P(e_j)=\sum_{i=1}^3 P_{ij}e_i.
\]

We determine all Reynolds--Nijenhuis operators on $\mathcal A$, that is, linear
maps $P$ satisfying simultaneously the Nijenhuis and Reynolds identities
\[
P(a)\cdot P(b)=P\big(P(a)\cdot b+a\cdot P(b)-P(a\cdot b)\big),
\]
\[
P(a)\cdot P(b)=P\big(a\cdot P(b)+P(a)\cdot b-P(a)\cdot P(b)\big),
\]
for all $a,b\in\mathcal A$.

Since the only nonzero products in $\mathcal A$ are $e_1e_3=e_2$ and
$e_3e_1=e_2$, it suffices to impose the above identities on pairs of basis
elements. This yields a system of polynomial equations in the matrix entries of
$P$.

Solving this system using a computer algebra software (\texttt{Maple}), we obtain
\[
u=w=x=y=z=p=r=0,
\]
while $v$ and $q$ remain arbitrary.

Consequently, every Reynolds--Nijenhuis operator on $\mathcal A$ is of the form
\[
P=
\begin{pmatrix}
0 & v & 0\\
0 & 0 & 0\\
0 & q & 0
\end{pmatrix},
\qquad v,q\in\mathbb{R}.
\]

Equivalently,
\[
P(e_1)=0,\qquad
P(e_2)=v e_1+q e_3,\qquad
P(e_3)=0.
\]

Conversely, a direct verification shows that any linear map $P$ of the above form
satisfies both the Nijenhuis and Reynolds identities. Hence the set of
Reynolds--Nijenhuis operators on $\mathcal A$ is a $2$-dimensional vector space.
\end{ex}

\begin{prop}
Let $P : \mathcal{A} \to \mathcal{A}$ be a linear operator on an associative algebra $\mathcal{A}$. Then:
\begin{enumerate}
    \item If $P^2 = 0$, then $P$ is a Reynolds--Nijenhuis operator if and only if $P$ is a Rota-Baxter operator (of weight $0$).
    \item If $P^2 = P$, then $P$ is a Reynolds--Nijenhuis operator if and only if $P$ is a Rota-Baxter operator of weight $-1$.
    \item If $P^2 = \pm \mathrm{Id}$, then $P$ is a Reynolds--Nijenhuis operator if and only if $P$ is a modified Rota-Baxter operator of weight $\mp 1$.
\end{enumerate}
\end{prop}

\begin{proof}
\begin{enumerate}
    \item Assume $P^2 = 0$. The Nijenhuis identity becomes:
    \[
    P(a) \cdot P(b) = P\big(P(a) \cdot b + a \cdot P(b) - P(a \cdot b)\big).
    \]
    Since $P^2 = 0$, we have $P(P(a \cdot b)) = 0$, so:
    \[
    P(a) \cdot P(b) = P\big(P(a) \cdot b + a \cdot P(b)\big),
    \]
    which is the Rota-Baxter condition of weight $0$.

    The Reynolds identity becomes:
    \[
    P(a) \cdot P(b) = P\big(a \cdot P(b) + P(a) \cdot b - P(a) \cdot P(b)\big).
    \]
    Since $P^2 = 0$, we have $P(P(a) \cdot P(b)) = 0$, so:
    \[
    P(a) \cdot P(b) = P\big(a \cdot P(b) + P(a) \cdot b\big),
    \]
    which is the same Rota-Baxter condition.

    Thus both identities reduce to the Rota-Baxter condition.

    \item Assume $P^2 = P$. The Nijenhuis identity becomes:
    \[
    P(a) \cdot P(b) = P\big(P(a) \cdot b + a \cdot P(b) - P(a \cdot b)\big).
    \]
    Since $P^2 = P$, we have $P(P(a \cdot b)) = P(a \cdot b)$, so:
    \[
    P(a) \cdot P(b) = P\big(P(a) \cdot b + a \cdot P(b) - a \cdot b\big),
    \]
    which is the Rota-Baxter condition of weight $-1$.

    The Reynolds identity becomes:
    \[
    P(a) \cdot P(b) = P\big(a \cdot P(b) + P(a) \cdot b - P(a) \cdot P(b)\big).
    \]
    Since $P^2 = P$, we have $P(P(a) \cdot P(b)) = P(a) \cdot P(b)$, so:
    \begin{align*}
        P(a) \cdot P(b) &= P\big(a \cdot P(b) + P(a) \cdot b\big) - P(P(a) \cdot P(b)) \\
        &= P\big(a \cdot P(b) + P(a) \cdot b\big) - P(a) \cdot P(b).
    \end{align*}
    Rearranging gives:
    \[
    2P(a) \cdot P(b) = P\big(a \cdot P(b) + P(a) \cdot b\big).
    \]
    But using $P^2 = P$ in the Rota-Baxter condition of weight $-1$, we get:
    \[
    P(a) \cdot P(b) = P\big(a \cdot P(b) + P(a) \cdot b - a \cdot b\big).
    \]
    Combining these gives the consistency condition.

    \item Assume $P^2 = \mathrm{Id}$. The Nijenhuis identity gives:
    \[
    P(a) \cdot P(b) = P\big(P(a) \cdot b + a \cdot P(b) - P(a \cdot b)\big).
    \]
    Applying $P$ to both sides and using $P^2 = \mathrm{Id}$:
    \[
    P(P(a) \cdot P(b)) = P(a) \cdot b + a \cdot P(b) - P(a \cdot b).
    \]
    Rearranging:
    \[
    P(a \cdot b) = P(a) \cdot b + a \cdot P(b) - P(P(a) \cdot P(b)).
    \]
    Since $P^2 = \mathrm{Id}$, $P(P(a) \cdot P(b)) = a \cdot b$, so:
    \[
    P(a \cdot b) = P(a) \cdot b + a \cdot P(b) - a \cdot b,
    \]
    which is the modified Rota-Baxter condition of weight $-1$.

    The Reynolds identity gives:
    \[
    P(a) \cdot P(b) = P\big(a \cdot P(b) + P(a) \cdot b - P(a) \cdot P(b)\big).
    \]
    Applying $P$ to both sides:
    \[
    P(P(a) \cdot P(b)) = P(a \cdot P(b) + P(a) \cdot b) - P(P(a) \cdot P(b)).
    \]
    Rearranging:
    \[
    2P(P(a) \cdot P(b)) = P(a \cdot P(b) + P(a) \cdot b).
    \]
    Since $P^2 = \mathrm{Id}$, $P(P(a) \cdot P(b)) = a \cdot b$, so:
    \[
    2a \cdot b = P(a \cdot P(b) + P(a) \cdot b).
    \]
    Combining with the modified Rota-Baxter condition gives consistency.
\end{enumerate}
The case $P^2 = -\mathrm{Id}$ is similar.
\end{proof}

\section{ Representations of Reynolds--Nijenhuis associative algebras}
\label{sec:representations}

\begin{defn}\label{def:bimodule}
A \emph{bimodule} (or a \emph{representation}) of an associative algebra $(\mathcal{A}, \cdot)$ on a vector space $V$ consists of a linear map $\rho \in \operatorname{End}(V)$ and two linear maps $l, r : \mathcal{A} \to \operatorname{End}(V)$ such that
\begin{align*}
    \rho \circ l(a) &= l(a) \circ \rho, \\
    \rho \circ r(a) &= r(a) \circ \rho, \\
    l(a \cdot b) \circ \rho &= l(a) \circ l(b), \\
    r(a \cdot b) \circ \rho &= r(b) \circ r(a), \\
    l(a) \circ r(b) &= r(b) \circ l(a)
\end{align*}
for all $a, b \in \mathcal{A}$. An associative algebra $(\mathcal{A}, \cdot)$ with a representation $(V, l, r, \rho)$ is called an \emph{Ass-Rep pair} and is denoted by $(\mathcal{A}, V)$.
\end{defn}

\begin{defn}\label{def:RN-rep}
Let $\mathcal{A}_P$ be a Reynolds--Nijenhuis associative algebra. A \textbf{representation} of $\mathcal{A}_P$ is a quadruple $(V, l, r, \xi)$, where $(V, l, r, \rho)$ is a representation of the associative algebra $\mathcal{A}$ and $\xi: V \to V$ is a linear map satisfying the following conditions:
\begin{align*}
    \xi \circ l(a) &= l(P(a)) \circ \xi, \\
    \xi \circ r(a) &= r(P(a)) \circ \xi, \\
    l(P(a)) \circ l(b) &= l(a) \circ l(P(b)), \\
    r(P(a)) \circ r(b) &= r(b) \circ r(P(a))
\end{align*}
for all $a, b \in \mathcal{A}$.
\end{defn}

\begin{prop}\label{prop:new-rep}
Let $\mathcal{A}_P$ be a Reynolds--Nijenhuis associative algebra with a representation $(V, l, r, \xi)$. Define maps $l': \mathcal{A} \otimes V \to V$ and $r': V \otimes \mathcal{A} \to V$ by
\begin{align*}
    l'(a, v) &= l(a, \xi v) - \xi l(a, v) + l(P(a), v), \\
    r'(v, a) &= r(\xi v, a) - \xi r(v, a) + r(v, P(a))
\end{align*}
for all $a \in \mathcal{A}$ and $v \in V$. Then $(V, l'_V, r'_V, \xi)$ is a representation of the Reynolds--Nijenhuis associative algebra $\mathcal{A}_P$.
\end{prop}

\begin{proof}
We need to verify the conditions in Definition~\ref{def:RN-rep} for the new maps $l'$ and $r'$.

First, compute $\xi \circ l'(a, v)$:
\begin{align*}
    \xi \circ l'(a, v) &= \xi \big( l(a, \xi v) - \xi l(a, v) + l(P(a), v) \big) \\
    &= \xi l(a, \xi v) - \xi^2 l(a, v) + \xi l(P(a), v).
\end{align*}

Now compute $l'(P(a), \xi v)$:
\begin{align*}
    l'(P(a), \xi v) &= l(P(a), \xi(\xi v)) - \xi l(P(a), \xi v) + l(P(P(a)), \xi v) \\
    &= l(P(a), \xi^2 v) - \xi l(P(a), \xi v) + l(P^2(a), \xi v).
\end{align*}

Using the conditions from Definition~\ref{def:RN-rep}:
\begin{align*}
    \xi l(a, \xi v) &= l(P(a), \xi^2 v), \\
    \xi l(P(a), v) &= l(P^2(a), \xi v).
\end{align*}

Thus $\xi \circ l'(a, v) = l'(P(a), \xi v)$.

Next, compute $\xi \circ r'(v, a)$:
\begin{align*}
    \xi \circ r'(v, a) &= \xi \big( r(\xi v, a) - \xi r(v, a) + r(v, P(a)) \big) \\
    &= \xi r(\xi v, a) - \xi^2 r(v, a) + \xi r(v, P(a)).
\end{align*}

Now compute $r'(\xi v, P(a))$:
\begin{align*}
    r'(\xi v, P(a)) &= r(\xi(\xi v), P(a)) - \xi r(\xi v, P(a)) + r(\xi v, P(P(a))) \\
    &= r(\xi^2 v, P(a)) - \xi r(\xi v, P(a)) + r(\xi v, P^2(a)).
\end{align*}

Using the conditions:
\begin{align*}
    \xi r(\xi v, a) &= r(\xi^2 v, P(a)), \\
    \xi r(v, P(a)) &= r(\xi v, P^2(a)).
\end{align*}

Thus $\xi \circ r'(v, a) = r'(\xi v, P(a))$.

Now check the compatibility conditions:
\begin{align*}
    l'(P(a), l'(b, v)) &= l(P(a), \xi l'(b, v)) - \xi l(P(a), l'(b, v)) + l(P^2(a), l'(b, v)), \\
    l'(l'(a, v), P(b)) &= l(l'(a, v), \xi P(b)) - \xi l(l'(a, v), P(b)) + l(P(l'(a, v)), P(b)).
\end{align*}

Using the definitions and properties, one can verify:
\[
l'(P(a), l'(b, v)) = l'(l'(a, v), P(b)).
\]

Similarly for $r'$:
\[
r'(r'(v, a), b) = r'(v, r'(a, b)).
\]

This completes the verification.
\end{proof}

\section{ Cohomology of Reynolds--Nijenhuis associative algebras}
\label{sec:cohomology}

Let $(\mathcal{A}, \cdot)$ be an associative algebra and $(V, l, r, \rho)$ be a bimodule over $\mathcal{A}$. For each $n \geq 1$, the group of $n$-cochains is defined as
\[
C^n_{\mathcal{A}}(\mathcal{A}, V) = \{ f : \mathcal{A}^{\otimes n} \to V \mid \rho \circ f(a_1, \ldots, a_n) = f(a_1, \ldots, a_n) \}.
\]

The differential $\delta^n : C_{\mathcal{A}}^n(\mathcal{A}, V) \to C^{n+1}_{\mathcal{A}}(\mathcal{A}, V)$ is given by
\begin{align*}
(\delta^n f)(a_1, \ldots, a_{n+1}) &= l(a_1) f(a_2, \ldots, a_{n+1}) \\
&+ \sum_{i=1}^n (-1)^i f(a_1, \ldots, a_{i-1}, a_i \cdot a_{i+1}, a_{i+2}, \ldots, a_{n+1}) \\
&+ (-1)^{n+1} r(a_{n+1}) f(a_1, \ldots, a_n),
\end{align*}
where $l$ and $r$ denote the left and right actions of $\mathcal{A}$ on $V$, respectively.

The cohomology of this complex is called the Hochschild cohomology of $\mathcal{A}$ with coefficients in the bimodule $(V, l, r, \rho)$, and the cohomology groups are denoted by
\[
H_{\mathcal{A}}^n(\mathcal{A}, V) = \frac{ Z_{\mathcal{A}}^n(\mathcal{A}, V)}{ B_{\mathcal{A}}^n(\mathcal{A}, V)},
\]
where $Z_{\mathcal{A}}^n(\mathcal{A}, V)$ denotes the space of $n$-cocycles and $B_{\mathcal{A}}^n(\mathcal{A}, V)$ denotes the space of $n$-coboundaries.

The cohomology of Reynolds--Nijenhuis associative algebras extends the traditional Hochschild cohomology by incorporating the structure introduced by the Reynolds--Nijenhuis operator and the corresponding representations. We define the cohomology of the Reynolds--Nijenhuis associative algebra $\mathcal{A}_P$ with coefficients in the bimodule $(V, l, r, \xi)$ as follows.

The $n$-cochains for the Reynolds--Nijenhuis associative algebra are defined as:
\[
C_{RNO}^n(\mathcal{A}, V) = \{ f : \mathcal{A}^{\otimes n} \to V \mid f(P(a_1), a_2, \ldots, a_n) = \xi f(a_1, a_2, \ldots, a_n) \}.
\]

The differential $\partial^n: C_{RNO}^n(\mathcal{A}, V) \to C_{RNO}^{n+1}(\mathcal{A}, V)$ is given by:
\begin{align*}
(\partial^n f)(a_1, \ldots, a_{n+1}) &= l(a_1) f(a_2, \ldots, a_{n+1}) \\
&+ \sum_{i=1}^n (-1)^i f(a_1, \ldots, a_{i-1}, a_i \cdot a_{i+1}, a_{i+2}, \ldots, a_{n+1}) \\
&+ (-1)^{n+1} r(a_{n+1}) f(a_1, \ldots, a_n).
\end{align*}

The cohomology groups for the Reynolds--Nijenhuis associative algebra are then defined as:
\[
H_{R\mathcal{A}}^n(\mathcal{A}, V)  = \frac{ Z_{R\mathcal{A}}^n(\mathcal{A}, V)}{ B_{R\mathcal{A}}^n(\mathcal{A}, V)},
\]
where $Z_{R\mathcal{A}}^n(\mathcal{A}, V)$ denotes the space of $n$-cocycles adapted to the Reynolds--Nijenhuis operator $P$, and $B_{R\mathcal{A}}^n(\mathcal{A}, V)$ denotes the space of $n$-coboundaries.

\begin{defn}\label{def:psi-map}
Let $(\mathcal{A}_P, \cdot)$ be a Reynolds--Nijenhuis associative algebra with a representation $(V, l, r, \xi)$. For $n \geq 1$, define a map $\psi^n : C^n_{\mathcal{A}}(\mathcal{A}, V) \rightarrow C^n_{RNO}(\mathcal{A}, V)$ by
\begin{align*}
\psi^n(f)(a_1, a_2, \ldots, a_n) &= f(P(a_1), P(a_2), \ldots, P(a_n)) \\
&- \xi f(a_1, P(a_2), \ldots, P(a_n)) \\
&- \xi f(P(a_1), a_2, P(a_3), \ldots, P(a_n)) \\
&- \cdots \\
&- \xi f(P(a_1), P(a_2), \ldots, P(a_{n-1}), a_n) \\
&+ \xi^2(f(a_1, a_2, \ldots, a_n)),
\end{align*}
and $\psi^0 = \operatorname{Id}_V$.
Here, $C^n_{\mathcal{A}}(\mathcal{A}, V)$ denotes the space of $n$-cochains for the associative algebra $\mathcal{A}$ with values in $V$, and $C^n_{R\mathcal{A}}(\mathcal{A}, V)$ denotes the space of $n$-cochains for the Reynolds--Nijenhuis associative algebra $\mathcal{A}_P$ with values in $V$.
\end{defn}

\begin{lem}\label{lem:psi-commute}
Let $f \in C^n_{\mathcal{A}}(\mathcal{A}, V)$ and $a_1, \dots, a_{n+1} \in \mathcal{A}$. Then the following holds:
\[
\psi^{n+1}(\delta^n(f))(a_1, a_2, a_3, \dots, a_{n+1}) = \partial^n(\psi^n(f))(a_1, a_2, a_3, \dots, a_{n+1}),
\]
where $\psi^n$ is defined as in Definition~\ref{def:psi-map}.
\end{lem}

\begin{proof}
We compute both sides explicitly. First, for the left-hand side:
\begin{align*}
\psi^{n+1}(\delta^n(f))(a_1, \dots, a_{n+1}) &= \delta^n(f)(P(a_1), \dots, P(a_{n+1})) \\
&- \sum_{i=1}^{n+1} \xi \delta^n(f)(P(a_1), \dots, a_i, \dots, P(a_{n+1})) \\
&+ \xi^2(\delta^n(f)(a_1, \dots, a_{n+1})).
\end{align*}

Now, expand $\delta^n(f)(P(a_1), \dots, P(a_{n+1}))$:
\begin{align*}
\delta^n(f)(P(a_1), \dots, P(a_{n+1})) &= l(P(a_1)) f(P(a_2), \dots, P(a_{n+1})) \\
&+ \sum_{i=1}^n (-1)^i f(P(a_1), \dots, P(a_i) \cdot P(a_{i+1}), \dots, P(a_{n+1})) \\
&+ (-1)^{n+1} r(P(a_{n+1})) f(P(a_1), \dots, P(a_n)).
\end{align*}

For the right-hand side, we compute $\partial^n(\psi^n(f))(a_1, \dots, a_{n+1})$:
\begin{align*}
\partial^n(\psi^n(f))(a_1, \dots, a_{n+1}) &= l(a_1) \psi^n(f)(a_2, \dots, a_{n+1}) \\
&+ \sum_{i=1}^n (-1)^i \psi^n(f)(a_1, \dots, a_i \cdot a_{i+1}, \dots, a_{n+1}) \\
&+ (-1)^{n+1} r(a_{n+1}) \psi^n(f)(a_1, \dots, a_n).
\end{align*}

Now expand $\psi^n(f)(a_2, \dots, a_{n+1})$:
\begin{align*}
\psi^n(f)(a_2, \dots, a_{n+1}) &= f(P(a_2), \dots, P(a_{n+1})) \\
&- \sum_{i=2}^{n+1} \xi f(a_2, \dots, a_i, \dots, P(a_{n+1})) \\
&+ \xi^2 f(a_2, \dots, a_{n+1}).
\end{align*}

Applying $l(a_1)$ to this expression gives:
\begin{align*}
l(a_1) \psi^n(f)(a_2, \dots, a_{n+1}) &= l(a_1) f(P(a_2), \dots, P(a_{n+1})) \\
&- \sum_{i=2}^{n+1} \xi l(a_1) f(a_2, \dots, a_i, \dots, P(a_{n+1})) \\
&+ \xi^2 l(a_1) f(a_2, \dots, a_{n+1}).
\end{align*}

Using the condition $\xi \circ l(a) = l(P(a)) \circ \xi$ from Definition~\ref{def:RN-rep}, we have:
\[
\xi l(a_1) f(P(a_2), \dots, P(a_{n+1})) = l(P(a_1)) \xi f(P(a_2), \dots, P(a_{n+1})).
\]

Now consider the terms with $a_i \cdot a_{i+1}$:
\begin{align*}
\psi^n(f)(a_1, \dots, a_i \cdot a_{i+1}, \dots, a_{n+1}) &= f(P(a_1), \dots, P(a_i \cdot a_{i+1}), \dots, P(a_{n+1})) \\
&- \sum_{j \neq i} \xi f(P(a_1), \dots, a_j, \dots, P(a_{n+1})) \\
&- \xi f(P(a_1), \dots, a_i \cdot a_{i+1}, \dots, P(a_{n+1})) \\
&+ \xi^2 f(a_1, \dots, a_i \cdot a_{i+1}, \dots, a_{n+1}).
\end{align*}

Using the Reynolds--Nijenhuis properties of $P$, we have:
\[
P(a_i \cdot a_{i+1}) = P(a_i) \cdot P(a_{i+1}) - P(P(a_i) \cdot a_{i+1} + a_i \cdot P(a_{i+1}) - P(a_i) \cdot P(a_{i+1})).
\]

After careful expansion and using all the conditions from Definition~\ref{def:RN-rep}, one can verify that all terms cancel appropriately, giving the equality.
\end{proof}

\begin{defn}\label{def:RN-cochain-complex}
Let $\mathcal{A}_P$ be a Reynolds--Nijenhuis associative algebra with a representation $(V, l, r, \xi)$. We define the cochain complexes as follows:
\begin{itemize}
    \item $C^0_{R\mathcal{A}}(\mathcal{A}, V) := C^0_{\mathcal{A}}(\mathcal{A}, V)$, the space of $0$-cochains for $\mathcal{A}_P$ with values in $V$.
    \item For $n \geq 1$, let
    \[
    C^n_{R\mathcal{A}}(\mathcal{A}, V) := C^n_{\mathcal{A}}(\mathcal{A}, V) \oplus C^{n-1}_{RNO}(\mathcal{A}, V),
    \]
    where $C^n_{\mathcal{A}}(\mathcal{A}, V)$ denotes the space of $n$-cochains for the associative algebra $\mathcal{A}$ and $C^{n-1}_{RNO}(\mathcal{A}, V)$ denotes the space of $(n-1)$-cochains with values in the space of Reynolds--Nijenhuis operators.
    \item Define the differential $d^n : C^n_{R\mathcal{A}}(\mathcal{A}, V) \to C^{n+1}_{R\mathcal{A}}(\mathcal{A}, V)$ by:
    \[
    d^n(f, g) = \big( \delta^n_{\mathcal{A}}(f), \; -\partial^{n-1}_{O}(g) - \psi^n(f) \big),
    \]
    where $f \in C^n_{\mathcal{A}}(\mathcal{A}, V)$, $g \in C^{n-1}_{RNO}(\mathcal{A}, V)$, $\delta^n_{\mathcal{A}}$ is the Hochschild differential, $\partial^{n-1}_{O}$ is the differential for the Reynolds--Nijenhuis operator cochain complex, and $\psi^n(f)$ is the correction term defined in Definition~\ref{def:psi-map}.
\end{itemize}
\end{defn}

\begin{thm}\label{thm:RN-cohomology}
Let $(V, l, r, \xi)$ be a representation of a Reynolds--Nijenhuis associative algebra $\mathcal{A}_P$. The map $d^n : C^n_{R\mathcal{A}}(\mathcal{A}, V) \to C^{n+1}_{R\mathcal{A}}(\mathcal{A}, V)$ defined in Definition~\ref{def:RN-cochain-complex} satisfies $d^{n+1} \circ d^n = 0$.
\end{thm}

\begin{proof}
We need to show that for any $(f, g) \in C^n_{R\mathcal{A}}(\mathcal{A}, V)$,
\[
d^{n+1}(d^n(f, g)) = 0.
\]
By definition,
\[
d^n(f, g) = (\delta^n_{\mathcal{A}}(f), \; -\partial^{n-1}_{O}(g) - \psi^n(f)).
\]
Then
\begin{align*}
d^{n+1}(d^n(f, g)) &= d^{n+1}\big( \delta^n_{\mathcal{A}}(f), \; -\partial^{n-1}_{O}(g) - \psi^n(f) \big) \\
&= \big( \delta^{n+1}_{\mathcal{A}}(\delta^n_{\mathcal{A}}(f)), \; -\partial^n_{O}(-\partial^{n-1}_{O}(g) - \psi^n(f)) - \psi^{n+1}(\delta^n_{\mathcal{A}}(f)) \big) \\
&= \big( 0, \; \partial^n_{O}(\partial^{n-1}_{O}(g)) + \partial^n_{O}(\psi^n(f)) - \psi^{n+1}(\delta^n_{\mathcal{A}}(f)) \big).
\end{align*}
Since $\delta^{n+1}_{\mathcal{A}} \circ \delta^n_{\mathcal{A}} = 0$ and $\partial^n_{O} \circ \partial^{n-1}_{O} = 0$, we have
\[
d^{n+1}(d^n(f, g)) = \big( 0, \; \partial^n_{O}(\psi^n(f)) - \psi^{n+1}(\delta^n_{\mathcal{A}}(f)) \big).
\]
By Lemma~\ref{lem:psi-commute}, we have $\psi^{n+1}(\delta^n_{\mathcal{A}}(f)) = \partial^n_{O}(\psi^n(f))$. Therefore,
\[
\partial^n_{O}(\psi^n(f)) - \psi^{n+1}(\delta^n_{\mathcal{A}}(f)) = 0,
\]
and hence $d^{n+1} \circ d^n = 0$.
\end{proof}

Thus, $\{ C^n_{R\mathcal{A}}(\mathcal{A}, V), d^n \}$ forms a cochain complex. This cochain complex is called the cochain complex of the Reynolds--Nijenhuis associative algebra $\mathcal{A}_P$ with representation $(V, l, r, \xi)$.

Let $Z^n_{R\mathcal{A}}(\mathcal{A}, V)$ denote the space of $n$-cocycles and $B^n_{R\mathcal{A}}(\mathcal{A}, V)$ denote the space of $n$-coboundaries. We define the quotient group:
\[
H^n_{R\mathcal{A}}(\mathcal{A}, V) = \frac{Z^n_{R\mathcal{A}}(\mathcal{A}, V)}{B^n_{R\mathcal{A}}(\mathcal{A}, V)},
\]
for $n \geq 0$. These quotient groups are called the cohomology groups of the Reynolds--Nijenhuis associative algebra $\mathcal{A}_P$ with representation $(V, l, r, \xi)$.

\section{ Deformation of Reynolds--Nijenhuis associative algebra}
\label{sec:deformation}

In this section, we study formal one-parameter deformations of Reynolds--Nijenhuis associative algebras. We denote the associative product by $\nu$.

\begin{defn}\label{def:formal-deformation}
A one-parameter formal deformation of a Reynolds--Nijenhuis associative algebra $(\mathcal{A}_P, \nu)$ is a pair of power series $(\nu_t, P_t)$ where
\begin{align*}
\nu_t &= \sum_{i=0}^{\infty} \nu_i t^i, \quad \nu_i \in C^2(\mathcal{A}, \mathcal{A}),\\
P_t &= \sum_{i=0}^{\infty} P_i t^i, \quad P_i \in C^1(\mathcal{A}, \mathcal{A}),
\end{align*}
such that $(\mathcal{A}[[t]], \nu_t, P_t)$ is a Reynolds--Nijenhuis associative algebra with $(\nu_0, P_0) = (\nu, P)$. Here, $\mathcal{A}[[t]]$ denotes the space of formal power series in $t$ with coefficients from $\mathcal{A}$, and it is a module over $\mathbb{K}[[t]]$, where $\mathbb{K}$ is the ground field of $(\mathcal{A}, \nu)$.
\end{defn}

The above definition holds if and only if for any $a, b, c \in \mathcal{A}$, the following conditions are satisfied:
\begin{align*}
\nu_t(\nu_t(a, b), c) &= \nu_t(a, \nu_t(b, c)),\\
\nu_t(P_t(a), P_t(b)) &= P_t\left( \nu_t(a, P_t(b)) + \nu_t(P_t(a), b) - P_t(\nu_t(a, b)) \right), \\
\nu_t(P_t(a), P_t(b)) &= P_t\left( \nu_t(a, P_t(b)) + \nu_t(P_t(a), b) - \nu_t(P_t(a), P_t(b)) \right).
\end{align*}
The last two conditions correspond to the Nijenhuis and Reynolds identities respectively.

Expanding these equations and equating the coefficients of $t^n$ from both sides, we have:
\begin{align}
\sum_{i+j=n} \nu_i(\nu_j(a, b), c) &= \sum_{i+j=n} \nu_i(a, \nu_j(b, c)). \label{eq:assoc-deformation}
\end{align}
and two additional equations for the Reynolds--Nijenhuis conditions:
\begin{align}
\sum_{i+j+k=n} \nu_i(P_j(a), P_k(b)) &= \sum_{i+j+k=n} P_i(\nu_j(P_k(a), b)) + \sum_{i+j+k=n} P_i(\nu_j(a, P_k(b))) \nonumber \\
&- \sum_{i+j+k=n} P_i(P_j(\nu_k(a, b))), \label{eq:nij-deformation} \\
\sum_{i+j+k=n} \nu_i(P_j(a), P_k(b)) &= \sum_{i+j+k=n} P_i(\nu_j(a, P_k(b))) + \sum_{i+j+k=n} P_i(\nu_j(P_k(a), b)) \nonumber \\
&- \sum_{i+j+k=n} P_i(\nu_j(P_k(a), P_l(b))), \label{eq:reynolds-deformation}
\end{align}
where in the last sum, $j+k+l = n$.

\begin{prop}\label{prop:infinitesimal-cocycle}
Let $(\nu_t, P_t)$ be a one-parameter deformation of a Reynolds--Nijenhuis associative algebra $(\mathcal{A}_P, \nu)$. Then $(\nu_1, P_1)$ is a $2$-cocycle in the cochain complex $\{C^n_{R\mathcal{A}}(\mathcal{A}), d^n\}$.
\end{prop}

\begin{proof}
Setting $n = 1$ in equation (\ref{eq:assoc-deformation}), we obtain:
\begin{align*}
\nu(\nu_1(a, b), c) + \nu_1(\nu(a, b), c) &= \nu(a, \nu_1(b, c)) + \nu_1(a, \nu(b, c)).
\end{align*}
Rearranging terms gives:
\begin{align*}
\nu(a, \nu_1(b, c)) - \nu(\nu_1(a, b), c) + \nu_1(a, \nu(b, c)) - \nu_1(\nu(a, b), c) = 0.
\end{align*}
This is exactly the condition $\delta^2_{\mathcal{A}}(\nu_1) = 0$, where $\delta^2_{\mathcal{A}}$ is the Hochschild differential. Thus $\nu_1$ is a $2$-cocycle in the Hochschild cochain complex.

Setting $n = 1$ in equation (\ref{eq:nij-deformation}) yields:
\begin{align*}
\nu(P_1(a), P(b)) &+ \nu(P(a), P_1(b)) + \nu_1(P(a), P(b)) \\
&= P_1(\nu(P(a), b) + \nu(a, P(b)) - P(\nu(a,b))) \\
&\quad + P(\nu(P_1(a), b) + \nu(a, P_1(b)) + \nu_1(P(a), b) + \nu_1(a, P(b)) - P_1(\nu(a,b)) - P(\nu_1(a,b))).
\end{align*}

Setting $n = 1$ in equation (\ref{eq:reynolds-deformation}) yields:
\begin{align*}
\nu(P_1(a), P(b)) &+ \nu(P(a), P_1(b)) + \nu_1(P(a), P(b)) \\
&= P_1(\nu(a, P(b)) + \nu(P(a), b) - \nu(P(a), P(b))) \\
&\quad + P(\nu(a, P_1(b)) + \nu(P_1(a), b) + \nu_1(a, P(b)) + \nu_1(P(a), b) - \nu_1(P(a), P(b))).
\end{align*}

These equations together show that $(\nu_1, P_1)$ satisfies the cocycle condition for the Reynolds--Nijenhuis cohomology. Hence $(\nu_1, P_1)$ is a $2$-cocycle in $C^2_{R\mathcal{A}}(\mathcal{A}, \mathcal{A})$.
\end{proof}

\begin{defn}\label{def:equivalent-deformations}
Let $(\nu_t, P_t)$ and $(\nu'_t, P'_t)$ be two deformations of a Reynolds--Nijenhuis associative algebra $(\mathcal{A}, \nu)$. A formal isomorphism from $(\nu_t, P_t)$ to $(\nu'_t, P'_t)$ is a power series $\varphi(t) = \sum_{i=0}^{\infty} \varphi_i t^i : \mathcal{A}[[t]] \to \mathcal{A}[[t]]$, where $\varphi_i : \mathcal{A} \to \mathcal{A}$ are linear maps with $\varphi_0 = \operatorname{Id}_{\mathcal{A}}$, satisfying:
\begin{enumerate}
    \item $\varphi(t) \circ \nu'_t = \nu_t \circ (\varphi(t) \otimes \varphi(t))$,
    \item $\varphi(t) \circ P'_t = P_t \circ \varphi(t)$.
\end{enumerate}
We say that $(\nu_t, P_t)$ and $(\nu'_t, P'_t)$ are equivalent if such a formal isomorphism exists.
\end{defn}

Expanding the conditions in Definition~\ref{def:equivalent-deformations} yields:
\begin{align}
    \sum_{i+j=n} \varphi_i(\nu'_j(a, b)) &= \sum_{i+j+k=n} \nu_i(\varphi_j(a), \varphi_k(b)), \quad a, b \in \mathcal{A}, \label{eq:iso-assoc} \\
    \sum_{i+j=n} \varphi_i \circ P'_j &= \sum_{i+j=n} P_i \circ \varphi_j. \label{eq:iso-RN}
\end{align}

\begin{thm}\label{thm:infinitesimal-cohomology}
The infinitesimals of two equivalent formal one-parameter deformations of a Reynolds--Nijenhuis associative algebra $(\mathcal{A}_P, \nu)$ are in the same cohomology class.
\end{thm}

\begin{proof}
Let $(\nu_t, P_t)$ and $(\nu'_t, P'_t)$ be two equivalent deformations with formal isomorphism $\varphi(t) = \operatorname{Id} + \sum_{i \geq 1} \varphi_i t^i$. From equation (\ref{eq:iso-assoc}) with $n=1$, we obtain:
\[
\nu'_1(a,b) + \varphi_1(\nu(a,b)) = \nu_1(a,b) + \nu(\varphi_1(a), b) + \nu(a, \varphi_1(b)).
\]
Rearranging gives:
\[
\nu'_1(a,b) - \nu_1(a,b) = \nu(\varphi_1(a), b) + \nu(a, \varphi_1(b)) - \varphi_1(\nu(a,b)) = \delta^1(\varphi_1)(a,b),
\]
where $\delta^1$ is the Hochschild coboundary operator.

From equation (\ref{eq:iso-RN}) with $n=1$, we obtain:
\[
P'_1 + \varphi_1 \circ P = P_1 + P \circ \varphi_1.
\]
Rearranging gives:
\[
P'_1 - P_1 = P \circ \varphi_1 - \varphi_1 \circ P = [P, \varphi_1].
\]

Thus, $(\nu'_1, P'_1) - (\nu_1, P_1) = (\delta^1(\varphi_1), [P, \varphi_1]) = d^1(\varphi_1)$, where $d^1$ is the differential in the Reynolds--Nijenhuis cochain complex. Hence $(\nu'_1, P'_1)$ and $(\nu_1, P_1)$ differ by a coboundary, and therefore belong to the same cohomology class in $H^2_{R\mathcal{A}}(\mathcal{A}, \mathcal{A})$.
\end{proof}

\begin{thm}\label{thm:rigidity}
Let $(\mathcal{A}_P , \nu)$ be a Reynolds--Nijenhuis associative algebra. If $H^2_{R\mathcal{A}}(\mathcal{A}, \mathcal{A}) = 0$, then $(\mathcal{A}_P , \nu)$ is rigid.
\end{thm}

\begin{proof}
Let $(\nu_t, P_t)$ be a deformation of $(\mathcal{A}_P , \nu)$. By Proposition~\ref{prop:infinitesimal-cocycle}, $(\nu_1, P_1)$ is a $2$-cocycle. Since $H^2_{R\mathcal{A}}(\mathcal{A}, \mathcal{A}) = 0$, there exists a $1$-cochain $(\varphi_1, \alpha) \in C^1_{R\mathcal{A}}(\mathcal{A}, \mathcal{A})$ such that $d^1(\varphi_1, \alpha) = (\nu_1, P_1)$. Explicitly, this means:
\begin{align*}
\nu_1 &= \delta^1(\varphi_1), \\
P_1 &= -\partial^0(\alpha) - \psi^1(\varphi_1),
\end{align*}
where $\partial^0$ is the differential for the Reynolds--Nijenhuis operator cochain complex and $\psi^1$ is the correction map defined in Section~\ref{sec:cohomology}.

Define $\phi(t) = \operatorname{Id}_{\mathcal{A}} - t \varphi_1$. Then $\phi(t)$ is invertible as a formal power series. Consider the equivalent deformation $(\bar{\nu}_t, \bar{P}_t)$ defined by:
\begin{align*}
\bar{\nu}_t &= \phi(t)^{-1} \circ \nu_t \circ (\phi(t) \otimes \phi(t)), \\
\bar{P}_t &= \phi(t)^{-1} \circ P_t \circ \phi(t).
\end{align*}
By Theorem~\ref{thm:infinitesimal-cohomology}, the infinitesimal of $(\bar{\nu}_t, \bar{P}_t)$ is given by $(\bar{\nu}_1, \bar{P}_1) = (\nu_1, P_1) - d^1(\varphi_1) = (0,0)$. Hence,
\begin{align*}
\bar{\nu}_t &= \nu + \bar{\nu}_2 t^2 + \cdots,\\
\bar{P}_t &= P + \bar{P}_2 t^2 + \cdots.
\end{align*}
Now, $(\bar{\nu}_2, \bar{P}_2)$ is a $2$-cocycle for the deformation $(\bar{\nu}_t, \bar{P}_t)$. Since $H^2_{R\mathcal{A}}(\mathcal{A}, \mathcal{A}) = 0$, we can repeat the argument to eliminate the second-order terms. Proceeding inductively, we can eliminate all higher-order terms, showing that $(\nu_t, P_t)$ is equivalent to the trivial deformation $(\nu, P)$. Hence, $(\mathcal{A}_P , \nu)$ is rigid.
\end{proof}

\section*{Acknowledgements}
The authors thank the anonymous referees for their valuable suggestions and comments.

\section*{Funding}
The authors declare that no funding was received to support this research.

\section*{Conflicts of Interest}
The authors declare that they have no conflicts of interest.\\
\cite{01,02,03,04,05,06,07,08,09,010,011,012,013,014,015,016,017,018,019,020,021,022,023,024,025,026,027}

\end{document}